\documentclass[10pt]{amsart}
\usepackage{hyperref}
\usepackage[bottom]{footmisc}
\usepackage{geometry}
\usepackage{xcolor}
\usepackage{graphicx}
\usepackage{enumitem}                               
\usepackage{amssymb}
\usepackage{mathrsfs}
\usepackage{amsthm}

\newtheorem{theorem}{Theorem}

\newtheorem{theoremn}{Theorem}[section] 
\newtheorem{lemman}[theoremn]{Lemma}

\newtheorem*{theorem*}{Theorem}
\newtheorem*{lemma*}{Lemma}

\theoremstyle{definition}

\newtheoremstyle{definition*}
{\topsep}
{\topsep}
{}
{0pt}
{\bfseries}
{.}
{ }
{\thmname{#1}\thmnumber{ #2}\thmnote{ (#3)}}

\theoremstyle{definition*}
\newtheorem*{definition*}{Definition}
\newtheorem*{remark*}{Remark}
\newtheorem*{claim*}{Claim}
\newtheorem*{corollary*}{Corollary}

\begin{document}
	\baselineskip 13pt
	
	\begin{abstract}   Let $G$ and $A$ be groups where $A$ acts on $G$ by automorphisms. We
		say ``\textit{the action of $A$ on $G$ is good}" if the equality $%
		H=[H,B]C_H(B)$ holds for any subgroup $B$ of $A$ and for any $B$-invariant
		subgroup $H$ of $G$. It is straightforward that every coprime action is a good action. In the present work we extend some results due to Ward, Gross, Shumyatsky, Jabara, and Meng and Guo under coprime action to good action. 
	\end{abstract}
	
	\title{extensions of several coprime results\\ to good action case}
	\author{G\"{u}l\. In Ercan$^{*}$}
	\address{G\"{u}l\. In Ercan, Department of Mathematics, Middle East
		Technical University, Ankara, Turkey}
	\email{ercan@metu.edu.tr}
	\thanks{$^{*}$Corresponding author}
	\author{\.{I}sma\. Il \c{S}. G\"{u}lo\u{g}lu}
	\address{\.{I}sma\. Il \c{S}. G\"{u}lo\u{g}lu, Department of Mathematics, Do%
		\u{g}u\c{s} University, Istanbul, Turkey}
	\email{iguloglu@dogus.edu.tr}
	\author{Enrico Jabara}
	\address{Enrico Jabara, Dipartimento di Filosofia, Universit\`a Ca' Foscari
		di Venezia, Venice, Italy}
	\email{jabara@unive.it}
	
	\maketitle
	
	\section{Introduction}
	
	Throughout all groups are finite, and the notation is standard. Let a group $A$ act on the group $G$ by automorphisms. We
	say ``\textit{the action of $A$ on $G$ is good}" if the equality $%
	H=[H,B]C_H(B)$ holds for any subgroup $B$ of $A$ and for any $B$-invariant
	subgroup $H$ of $G$. This concept is introduced in \cite{egj} as a generalization of coprime action, namely, the case where $(|G|,|A|)=1$. As the first work on good action, \cite{egj} is essentially devoted to  extensions of some coprime results due to Turull obtained in \cite{tur} and \cite{turull}. In the present paper we emphasize the importance of ``good action" once more by extending to good action case the main results of \cite{gross}, \cite{wa}, \cite{shum}, \cite{ja}, \cite{guo} which are proven under the coprimeness assumption. The key result leading to these new observations is the following.
	
	\begin{lemman}\label{lem1}
		Let $R$ be an $r$-group and let $A$ be a noncyclic abelian $p$-group acting
		faithfully on $R$. If this action is good, then 
		\begin{equation*}
			R= \big \langle \, C_R(a)\,\,:\,\,1\ne a\in A\,\, \big \rangle.
		\end{equation*}
	\end{lemman}
	
	\begin{proof}
		Observe that the action is trivial by Proposition 2.5 of \cite{egj} when $r=p$ and the claim follows. The
		result is well known in case where $r\ne p.$
	\end{proof}
	
	The following result can be regarded as the main theorem of this paper. It generalizes \cite{shum} to the case of a good action the proof of which is partially independent of the method used in \cite{shum}.  \\
	
\begin{theorem}\label{A}
	Let $p$ be a prime, $n$ a positive integer. Suppose that $G$ is a finite 
		solvable group acted on by an elementary abelian $p$-group $A$ with $|A|\geq
		p^{n+1}$.  If this action is good and $C_G (a)$ is of Fitting height at most 
		$n$ for every  nontrivial element $a$ of $A$ then $G$ is of Fitting  height
		at most $n + 1.$ Moreover, if $|A|\geq p^{n+2}$ then $G$ is of Fitting  height
		at most $n.$
		\end{theorem}
	
	The next result is obtained as an extension of Theorem 3.3 in \cite{gross}. It is achieved by applying the same argument
	as in \cite{gross} by the use of Theorem 4.5 of \cite{egj} and Lemma \ref{lem1} and Theorem \ref{A}.\\

	\begin{theorem}\label{B}
	Let $p$ be a prime. Suppose that $G$ is a finite 
		solvable group acted on by an elementary abelian $p$-group $A$ with $|A|\geq
		p^{4}$. If this action is good and $C_G (a)$ is supersolvable for every  nontrivial element $a$ in $A$ then $G$ is supersolvable.
		\end{theorem}
	
	Similar to Theorem \ref{B} we extend the main theorem of \cite{guo} as follows.\\
	
	\begin{theorem}\label{C}  Let $p$, $r$ be two primes and suppose that an elementary abelian r-group $A$ of order $r^2$ acts  on a $p$-solvable group $G$ in such a way that $C_G(a)$ is $p$-nilpotent for each nonidentity $a\in A$. Then $G$ is $p$-nilpotent by $p$-nilpotent.
		\end{theorem}
	
	Using the main theorem of \cite{egj} once more we observe that a result due to Ward \cite{wa} can also be extended to good action case. Namely, we prove\\
	
	\begin{theorem}\label{D}Let a noncyclic abelian $p$-group $A$ act on the group $G$ so that the
		action is good. Suppose that for any prime $q$ dividing $|C_G(A)|$, the $q$%
		-elements of $C_G(A)$ centralize the $q^{\prime}$-elements of $C_G(a)$ for
		all nonidentity elements $a\in A$. Then $G$ is solvable, $[G,A]$ is a $p^{\prime}$-group, $G=NP$ where $N$ is an $A$-invariant Hall
		$p^{\prime}$-subgroup of $G$ and $P$ is an $A$-invariant Sylow $p$-subgroup of
		$G$. Moreover, if $A$ is elementary abelian of order $p^{n}$, then $G$ is of Fitting height at most $n$.
		\end{theorem}
	
	In the last section we prove the theorem below and give some further examples, namely, the extensions of the results of \cite{ja}.\\
	
	\begin{theorem}\label{E} Let a noncyclic abelian $p$-group
		$A$ of order $p^n$ act on the solvable group $G$ so that the
		action is good. Suppose that there exists a natural number $m$ such that $$%
		[C_G(a),\, C_G(A)]_{m}=[C_G(a),\underbrace{\, C_G(A),\ldots ,C_G(A)]}_{m-times}=1$$ for all $a\in A$. Then $G$ is of Fitting height at most $n$, and this bound is the best possible.
		\end{theorem}
	
	\section{Proof of Theorem A}
	
	\begin{lemman}\label{lem21}
		Let $A$ be an abelian group acting on the group $G$. Then for any proper subgroup $B$ of $A$ and for any $B$-invariant irreducible section $V$ of $G$ there exists $v\in V$ such that $C_B(V)=C_B(v)$, that is, $A$ acts with regular orbits on $G$.
	\end{lemman}
	
	\begin{proof}
		Let $0\ne v\in V$. Then $C_V(C_B(v))\ne 0$ is $B$-invariant, hence $C_B(V)=C_B(v)$ as required.
	\end{proof}
	
	Corollary 1.3 in \cite{shum} is valid when coprimeness condition is
	replaced by assumption that the action is good. Namely we have the following.
	
	\begin{lemman}\label{lem22} Let $G$ be a group on which an elementary abelian $p$-group $A$ with $|A|\geq
		p^{n+1}$ acts. Suppose that this action is good and that $G=\prod_{i=1}^hS_i$
		where $S_i$ are $A$-invariant $p_i$-subgroups such that $p_i\ne p_{i+1}$ and 
		$[S_i,S_{i+1}]=S_{i+1}$. If $C_G (a)$ is of Fitting height at most $n$ for every nontrivial element $a$ of $A$, then we have 
		\begin{equation*}
			S_q= \big \langle S_q\cap F(C_G(a)) : 1\ne a \in A \big\rangle
		\end{equation*}
		for any $q \geq n$.
	\end{lemman}
	
	\begin{proof}
		This can be achieved by applying the same argument as in \cite{shum} by the use of Lemma \ref{lem1} and of Theorem 4.5 of \cite{egj}.
	\end{proof}
	
	\begin{proof}[Proof of Theorem \ref{A}] 
	Let $h=h(G).$ Arguing as in \cite{egj} we can build an irreducible $A$-tower of height $h$ in $G$, that is a sequence of subgroups $S_i, \,i=1\ldots ,h,$ for which the following conditions are satisfied:
	\begin{itemize}
		\item[(1)] $S_i$ is a $p_i$-group, $p_i$ is a prime, for $i= 1,\ldots ,h$;
		
		\vspace{1mm}
		
		\item[(2)] $S_i$ normalizes $S_j$ for $i\leq j$;
		
		\vspace{1mm}
		
		\item[(3)] Set $P_h=S_h,\, P_i=S_i/T_i$ where $T_i=C_{S_i}(P_{i+1}), \, i= 1,\ldots ,h-1,$ and we assume that $P_i$ is not trivial for $i= 1,\ldots ,h$;
		
		\vspace{1mm}
		
		\item[(4)] $p_i\ne p_{i+1}, \, i= 1,\ldots ,h-1$.
		
		\vspace{1mm}
		
		\item[(5)] $\Phi(\Phi(P_i)) = 1,\, \Phi(P_i)\leq Z(P_i)$ and, if $p_i\ne 2$, then $P_i$ has exponent $p_i$ for $i= 1,\ldots ,h.$ Moreover $P_{i-1}$  centralizes $\Phi(P_i)$;
		
		\vspace{1mm}
		
		\item[(6)] $P_1$ is elementary abelian;
		
		\vspace{1mm}
		
		\item[(7)] There exists $H_i$ an elementary abelian $A$-invariant subgroup of $P_{i-1}$ such that $[H_i,P_i] = P_i$ for $i=2,\ldots, h$;
		
		\vspace{1mm}
		
		\item[(8)] $(\prod_{j=1}^{i-1}S_j)A$ acts irreducibly on $P_i/\Phi(P_i)$. 
	\end{itemize}
	
	Clearly we may assume that $h=n+2$ (resp. $h=n+1$) in case where $|A|\geq p^{n+1}$ (resp. $|A|\geq p^{n+2}$) and that $G=\prod_{i=1}^h S_i.$

	We begin with proving the first claim of the theorem. Set $P=S_{h-1}$, $X=\prod_{i=1}^{h-1} S_i$, and let $V$ denote the Frattini factor group of $P_{h}.$ By Fong-Swan theorem we may assume that $V$ is an irreducible complex $XA$-module. We shall proceed over the following steps:\newline
	
	$(1)$ $C_V(A)=0. $
	
	\begin{proof}
		We apply now Lemma \ref{lem22} to the pair $P,A$ and get 
		\begin{equation*}
			P= \big \langle P\cap F(C_G(a)) : 1\ne a \in A \big \rangle.
		\end{equation*}
		On the other hand 
		\begin{equation*}
			\big [C_{S_{h}}(A),P\cap F(C_G(a)) \big ]\leq \big [S_{h}\cap C_{G}(a), P\cap F(C_G(a)) \big ]=1
		\end{equation*}
		since $C_{S_{h}}(a)\leq O_{p_{h}}(C_G(a)) $ for each $a\in A$. Then $C_V(A)\leq C_V(P)=1$. 
	\end{proof}
	
	$(2)$ $C_A(P)=1$ and $(|P|,|A|)=1$.
	
	\begin{proof}
		
		We can observe that $A_1=C_A(P)$ centralizes all the subgroups $P, 
		S_{h-2},\ldots ,S_1$ due to good action: Firstly we have $ [P_{h-2}, A_1]=1$
		by the three subgroups lemma. Repeating the same argument we get $[P_{i}, A_1]=1$ for $i=1,\ldots ,h-2$. Since $C_{S_i}(A_1)T_i=S_{i}$ by
		Proposition 2.2 (3) in \cite{egj}, we may assume that $[S_i, A_1]=1$ for $i=1,\ldots ,h-1.$ It then follows that $h(C_G(a))\geq h-1$ for some $a\in A$%
		, which is impossible. Thus we have $C_A(P)=1.$ Notice that $[P,A_p]\leq [G,A_p]\cap P=1$ by  Proposition 2.5 in \cite{egj}. This shows that $(|P|,|A|)=1$ as claimed.
	\end{proof}
	
	\textit{(3) Theorem follows.}
	
	\begin{proof}
		Let now $M$ be an $X$ -homogeneous component of $V$ and let $ B=N_{A}(M).$
		Then $M$ is an irreducible $XB$-module such that $M|_X$ is homogeneous, and $C_{M}(B)=0$ as $C_{V}(A)=0.$
		
		We consider now the set of all pairs $(M_{\alpha}, C_{\alpha})$ such that $M_{\alpha}$ is an irreducible $XC_{\alpha}$-submodule  of $M_{_{XC_{\alpha}}}
		, \, M_{\alpha}|_{X}$ is homogeneous, and $C_{M_{\alpha}}(C_{\alpha})=0$. Choose $(M_1,C)  $ with $|C|$ minimum. Then $C_{M_1}(C_0)\ne 0$ for every $C_0<C$, $ (M_1)_{_{X}}$ is homogeneous and $Ker(X$ on $M_1)=Ker(X$ on $M).$
		
		Set now $\bar{X}=X/Ker(P$ on $M_1)$. We can observe that $[Z(\bar{P}),C]=1$.  Otherwise, it follows by Theorem 3.3 in \cite{egj} that for
		any $\bar{P}$-homogeneous component $W$ of $(M_1)_{_{\bar{P}}}$, the module $W$ is $C$-invariant and $ \bar{X}=N_{\bar{X}}(W)C_{\bar{X}}(C)$. Then $C_{%
			\bar{X}}(C)$ acts  transitively on the set of all $\bar{P}$-homogeneous components of $M_1$.  Clearly we have $[Z(\bar{P}),C]\leq Ker(\bar{P}$ on $W) $ and hence $[Z(\bar{P}),C]=1$, as claimed.
		
		Suppose now that $\bar{P}$ is abelian. Then $[\bar{P},C]=1$ by the above paragraph, which forces that $[\bar{X},C]=1$. Now, $h-1=h(\bar{X})\leq h(C_{\bar{X}}(C))\leq h-2$. This contradiction shows that $\bar{P}$ is nonabelian.
		
		Let now $U$ be a homogeneous component of $(M_1)_{_{\Phi(\bar{P})}}$. Notice that $\Phi(\bar{P})\leq Z(\bar{P})$ and so $[\Phi( \bar{P}),C]=1$. Then $U$
		is $C$-invariant. Set $\widehat{\bar{P}}=\bar{P} /Ker({\bar{P}}$ on $U)$.
		Now $\Phi(\widehat{\bar{P}})=\widehat{\Phi(\bar{P})}  $ is cyclic of prime
		order $p$. Since $[Z(\bar{P}),C]=1$ we get $[X,C]\leq  C_X(Z(\bar{P}))$ by the three subgroups lemma. Now clearly we have $ [X,C]\leq N_X(U).$ That is $X=N_X(U)C_X(C)$ as the action is good and so $ C_X(C)$ acts transitively on the set of all homogeneous components of $ (M_1)_{_{\Phi(\bar{P})}}$. Hence $M_1=\bigoplus_{t\in T}U^t$ where $T$ is a  transversal for $N_X(U)$ in $X$
		contained in $C_X(C).$ Notice that $N_{\bar{X}C}(U)=N_{\bar{X}}(U)C$. Set 
		$X_1=C_X(\Phi(\bar{P}))$. Now $C_{XC}(\Phi(\bar{P}))=X_1C\lhd XC$ and we have $[X,C]\leq X_1$ by the three subgroups lemma. Then $X=X_1C_X(C)$.
		Clearly we have $PS_{n}\leq X_1\leq N_X(U)$ and $ X_1C\lhd XC\lhd \lhd XA.$
		Recall that $P/\Phi(P)$ is an irreducible $XA$-module and hence $P/\Phi(P)$
		is completely reducible as an $X_1C$-module. Note that $\widehat{\bar{P}}
		/\Phi(\widehat{\bar{P}})\cong P/\Phi(P)C_P(U).$  As $P/\Phi(P)$ is
		completely reducible we see that so is $P/\Phi(P)C_P(U)$.  Hence $\widehat{
			\bar{P}}/\Phi(\widehat{\bar{P}})$ is also completely  reducible.
		
		Since $\widehat{\Phi(\bar{P})}\leq \widehat{Z(\bar{P})}$, there  is an $X_1C$-invariant subgroup $E$ containing $\widehat{\Phi(\bar{P})}$ so  that  
		\begin{equation*}
			\widehat{\bar{P}}/\widehat{\Phi(\bar{P})}=\widehat{Z(\bar{P})}/\widehat{%
				\Phi( \bar{P})}\oplus E/\widehat{\Phi(\bar{P})}.
		\end{equation*}
		Since the above sum is direct we have $\widehat{\Phi(\bar{P})}=\widehat{Z(\bar{P}%
			) }\cap E=Z(E).$ Thus we get $Z(E)=\widehat{
			\Phi(\bar{P})}=({\widehat{\bar{P}}})^{\prime}.$ As $E\unlhd \widehat{\bar{P}}$  we get $\Phi(E)\leq \widehat{\Phi(\bar{P})}=Z(E).$ It follows that $
		Z(E)=E^{\prime}=\Phi(E)=\widehat{\Phi(\bar{P})}$ is cyclic of prime order and hence $E$ is extraspecial. Now $[Z(\bar{P}),C]=1$ gives $[\widehat{Z(\bar{P})},C]=1$. Thus $[Z(E),C]=1.$
		
		Next we observe that $C_{C}(E)=1$: Otherwise there is a nonidentity  element 
		$a$ in $C$ such that $[\widehat{\bar{P}},a]=1$ and hence $ [\bar{P},a]\leq
		Ker(\bar{P}$ on $U)$. Since $X=X_1C_X(C)\leq N_X(U)C_X(C)$  we get $[\bar{P},a]\leq Ker(\bar{P}$ on $M)$, that is, $[\bar{P},a]=1$, which forces that $h-1=h(\bar{X})\leq h(C_G(a))\leq h-2$. This contradiction shows that $%
		C_{C}(E)=1$, as claimed.
		
		By $(2)$, $p$ is coprime to $|C|$. We apply now Lemma 2.1  in \cite{Esp} to
		the action of the semidirect product $EC$ on the  module $U$ and see that $C_U(C)\ne 0$. This final contradiction  completes the proof of the first claim of the theorem.

		Our proof of the second claim is essentially the same as in Theorem 3.3 in \cite{gross}: Assume that $|A|\geq p^{n+2}$. We may also assume that $h=n+1.$ Set $A_i=C_A(P_i)$ for $i=1,\ldots ,h$ and $A_0=A$. Clearly $A_i\leq A_{i-1}$ for $i=1, \ldots ,h.$ Notice that for each $a\in A$ we have either $C_{P_1}(a)=1$ or $[P_1, a]=1$ by the irreducibility of $P_1$ as an $A$-module. Also note that for $b\in A_{i-1}$ we have $[P_{i-1},b]=1$ whence $[P_j,b]=1$ for each $j<i$. Then we may assume that $[\,\,\prod_{j=1}^{i-1}S_j,b]=1$ and hence, by $(8)$, we have $C_{P_i}(b)=1$ for all  $b\in A_{i-1}\setminus A_i.$ On the other hand if $A_{i-1}/ A_i$ is noncylic, Lemma \ref{lem1} applied to the action of $A_{i-1}/ A_i$ yields that $P_i=\langle C_{P_i}(bA_i)\,\,:\,\,b\in A_{i-1}\setminus A_i\,\,\rangle.$ This contradiction shows that $|A_{i-1}/ A_i|\leq p$ for each $i=1,\ldots ,h=n+1$, that is, $|A|\leq p^{n+1}$ which is the final contradiction completing the proof of the second claim. 
	\end{proof}
	
	\section{Proof of Theorem B} 
	
	We shall need the following lemma which is also of independent interest too as an extension of Lemma 3.2 of \cite{gross} to good action case.
	
	\begin{lemman} \label{lem31} Suppose that $G$ is a finite solvable group acted on by an elementary abelian $p$-group $A$ with $|A|\geq
		p^{3}$. If this action is good and $C_G (a)$ is abelian for every  nonidentity $a\in A$ then $G$ is abelian.
		
	\end{lemman}	
	\begin{proof} We essentially follow the steps of the proof of Lemma 3.2 of \cite{gross}. Let $G$ be a minimal counterexample. Then $G'$ is a minimal $A$-invariant normal subgroup of $G$ . Note that the group $G$ is nilpotent by Theorem \ref{A}. It follows that $G' \cap Z(G)\ne 1$ and hence $G'\leq Z(G)$ by the minimality of $G'.$ Then we get $C_{G'}(a)$ is either trivial or equal to $G'$ for each nonidentity $a\in A$. Let now $C=C_A(G')$ and $B$ be a complement to $C$ in $A$. Notice that $\langle C_{G'}(b): 1\ne b \in B\rangle =1. $ This yields by Lemma \ref{lem1} that $B$ is cyclic and so $|C|\geq p^2.$ Applying Lemma\ref{lem1} we have $G=\langle C_{G}(a): 1\ne a \in C\rangle$.  Let $x$ and $y$ be two nonidentity elements of $C$. Observe that $[C_G(x),C_G(y),\langle x\rangle]=1=[\langle x\rangle, C_G(x), C_G(y)]$. It follows by the three subgroups lemma that $[ C_G(y), \langle x\rangle, C_G(x)]=1.$ Due to good action we have $C_G(y)=[C_G(y), \langle x\rangle]C_{C_G(y)}(x)$. Then $[C_G(y),C_G(x)]=[C_G(y), \langle x\rangle, C_G(x)]=1$. As a result, $G$ is abelian.
	\end{proof}
	
		\begin{proof}[Proof of Theorem \ref{B}]  We shall follow the steps of the proof of Theorem 3.3 in \cite{gross}. Let $G$ be a minimal counterexample to Theorem \ref{B}. We can observe that $F(G)$ is the unique minimal $A$-invariant normal subgroup of $G$ and is an elementary abelian $q$-group for some prime $q.$ By Theorem \ref{A} it follows that $G/F(G)$ is a nilpotent $q'$-group. Since $G$ is not supersolvable, the minimality of $G$ implies that $G=F(G)R$ where $R$ is an $A$-invariant $r$-subgroup of $G$, and either $R$ is nonabelian or the exponent of $R$ does not divide $q-1.$ Let $C=C_A(R)$ and $B$ be a complement to $C$ in $A$. Suppose first that $|C|$ is not cyclic. We see by Lemma \ref{lem1} that $C_{F(G)}(a)\ne 1$ for some nonidentity $a\in C$. Notice that $C_{F(G)}(a)$ is $RA$-invariant and hence is equal to $F(G)$ by the uniqueness of $F(G).$ This implies that $G=C_G(a)$ is supersolvable. Therefore $C$ is cyclic and so $|B|\geq p^3.$
	
	Let $1\ne b\in B$. Then $C_R(b)\ne R$ and hence $F(G)C_R(b)$ is a proper $A$-invariant subgroup of $G$. It follows that $F(G)C_R(b)$ is supersolvable which yields that $C_R(b)$ is abelian of exponent
	dividing $q-1$. Now $R$ is abelian by Lemma \ref{lem31}. As $R = \langle C_R(a) : 1\ne a\in B\rangle$ we see that the exponent of $R$ must divide $q-1$, establishing the claim. 
	\end{proof}
	
	\section{Proof of Theorem C} 
	We first prove some lemmas which will be used in the proof of Theorem \ref{C}.

	\begin{lemman} \label{lem41}Let $A$ be a noncylic abelian $r$-group acting on the group $H$ such that the action is good. Assume that $H=VG$ where $V$ and $G$ are both $A$-invariant, $V$ is a normal $p$-subgroup of $H$ with $C_G(V)=1$ and that $C_V(a)\leq C_V(g)$ for each nonidentity $a\in A$ and each $p'$-element $g\in C_G(a)$. Then either $p=r$ or $G$ is an $r'$-group. 
	\end{lemman}
	
	\begin{proof} Notice that we have $V=\langle \;C_V(a)\,\,:\,\,1\ne a\in A\,\,\rangle$ by Lemma \ref{lem1}. On the other hand, by hypothesis, the group $C_V(a)$ is centralized by each  $p'$-element of  $C_G(A)$ for each nonidentity $a\in A$. Due to faithful action of $G$ on $V$ we observe that $C_G(A)$ is a $p$-group. Let now $R$ be an $A$-invariant Sylow $r$-subgroup of $G$. Notice that if $R\ne 1$, then  $C_R(A)\ne 1$ whence $p=r$, as desired.
	\end{proof}
	
	\begin{lemman}\label{lem42} Let an elementary abelian $r$-group $A$ of order $r^2$ act on a $p$-solvable group $G$ and let $V$ be a faithful $GA$-module over a field $F$ of characteristic $p$, where $p\ne r$, and $O_p(G)=1$. Suppose that the action of $A$ on $VG$ is good and that\par
		
		$(i)$ $C_G(a)$ is $p$-nilpotent for each nonidentity $a\in A$;\par
		$(ii)$ $C_V(a)\leq C_V(g)$ for each nonidentity $a\in A$ and each $p'$-element $g\in C_G(a)$.
		
		\noindent Then $G$ is $p$-nilpotent.
	\end{lemman}
	
	\begin{proof} It can be easily seen that by Lemma \ref{lem41} we may assume that $A$ acts coprimely on $G$. Appealing to Theorem 12 of \cite{guo} we have the result. \end{proof}
	
		\begin{proof}[Proof of Theorem \ref{C}]  This can be achieved by repeating the proof of Theorem A in \cite{guo} word by word by replacing Theorem 12 of \cite{guo} with Lemma \ref{lem42} above.
			\end{proof}

	\section{Proof of Theorem D}

		We say that the pair $(A,G)$ satisfies \textbf{Hypothesis$(p)$} for a prime $p$ if $A$ is a
		group acting on the group $G$, the action of $A$ on $G$ is good, and every $p$-element of $C_{G}(A)$ centralizes every
		$p^{\prime}$-element of $C_{G}(a)$ for any $1\neq a\in A.$

	\begin{lemman} \label{lem52} If the pair $(A,G)$ satisfies Hypothesis$(p)$ then $C_{G}(A)=P\times O_{p^{\prime}}(C_{G}(A))$ where $P\in Syl_{p}(C_{G}(A)).$
\end{lemman}	
		\begin{proof}
			Let $P\in Syl_{p}(C_{G}(A)).$ By Hypothesis$(p)$ it holds that $[P,S]=1$ for any $S\in Syl_{s}(C_{G}(A))$ where $s$ is a prime different from $p$. This implies that $P$ is normal in
			$C_{G}(A)$. By the Schur-Zassenhaus theorem  there exists a subgroup
			$H$ of $C_{G}(A)$ such that $C_{G}(A)=PH$ and $P\cap H=1.$ Clearly then we have  $[P,H]=1$ which completes the proof$.$
		\end{proof}

	\begin{theoremn} \label{the53}
		Suppose that $A$ is a $p$-group and that the pair $(A,G)$ satisfies
		Hypothesis$(p)$. Then $G$ admits a normal $p$-complement.
	\end{theoremn}
	
	\begin{proof}
		Let $G$ be a minimal counterexample to the theorem and let $T$ be a Sylow
		$p$-subgroup of the semidirect product $GA$ containing $A.$ The subgroup
		$P=T\cap G$ is then an $A$-invariant Sylow $p$-subgroup of $G,$ and is
		contained in $C_{G}(A)$ since the action of $A$ on $G$ is good. By Lemma \ref{lem52}, $C_{G}(A)=P\times D$ where $D$ is a Hall $p^{\prime}$-subgroup of
		$C_{G}(A).$
		
		If $[G,A]<G$ then $[G,A]=O_{p^{\prime}}([G,A])(P\cap [G,A])$ by the minimality of $G$ 
		and hence $$G=[G,A]C_{G}(A)=O_{p^{\prime}}([G,A])(P\times D)=O_{p^{\prime}%
		}(G)P$$ which is not possible. So $[G,A]=G.$
		
		As $G$ has no normal $p$-complement, \cite[Theorem 5.26]{Misaacs} implies the existence of a nontrivial subgroup
		$K$ of $P$ such that $N_{G}(K)$ does not have a normal $p$-complement. Since $K$ is $A$-invariant we see that  $N_{G}(K)=G.$ As $G/K$ satisfies the hypothesis of
		the theorem it follows by the minimality of $G$ that $G/K$ has a normal $p$-complement, say $N/K$. If $N<G$
		then $N$ has a normal $p$-complement $M$ by the minimality of $G$. Clearly $M$ is the normal
		$p$-complement of $G.$ This forces that $G=N$ whence $K=P$ and $P$ is normal in
		$G=[G,A]$. By the three subgroups lemma we get that $P\leq Z(G)$. This leads by the
		Schur-Zassenhaus theorem to the the final contradiction that $G$ has a normal
		$p$-complement. 
	\end{proof}
	
	\bigskip

	\begin{proof}[Proof of Theorem \ref{D}] 	Let $G$ be a minimal counterexample to the theorem. By Theorem \ref{the53} we
			can assume that $G$ has a normal $p$-complement $N$. Then there exists a Sylow
			$p$-subgroup $P$ of $G$ which is centralized by $A$ and $G=NP.$ In particular
			$[G,A]\leq N$ and hence is a $p^{\prime}$-group. Furthermore as $A$ acts
			coprimely on $N$ we see that for any prime $q\neq p$ there exists an
			$A$-invariant Sylow $q$-subgroup $Q$ of $N.$ As $A$ is noncyclic, Lemma 1.1 yields 
			$Q=\left\langle C_{Q}(a):1\neq a\in A\right\rangle $ which implies that
			$[Q,R]=1$ for any Sylow $r$-subgroup $R$ of $C_{G}(A)$ with $r\neq q$. In
			particular we get $[N,P]=1$ which gives that $G=N\times P$. Clearly $G$ is solvable if $N$ is solvable
			which is the case if $N$ is proper in $G.$ Therefore we may	assume that $G$ is a $p'$-group which does not
			have any proper, nontrivial $A$-invariant normal subgroup. In particular $G=G_{1}\times G_{2}\times\cdots\times G_{n}$
			where $G_{1}$ is a nonabelian simple group and $G_{i}\cong G_{1},i=1,2,\ldots ,n$ and
			$A$ acts transitively on $\{G_{1},G_{2},\ldots ,G_{n}\}.$
			
			By \cite{Be} we may assume that $C_G(A)\ne 1$. Let $r$ be a prime dividing $\left\vert C_{G}(A)\right\vert $ and $q$ a prime different from $r.$ As $A$ acts coprimely on
			$G$ there exists an $A$-invariant Sylow q-subgroup $Q$ of $G.$ Let $R\in
			Syl_{r}C_{G}(A)$. By the argument
			in the above paragraph we see that $[R,Q]=1.$ If 1$\neq x\in R$ where
			$x=x_{1}x_{2}\cdots x_{n}$ with $x_i\in G_i,\, i=1,2,\ldots ,n,$ then $x_{1}\neq1$  as $A$ acts transitively on the set of
			components of $x$ and centralizes $x$. Notice that $Q\cap G_{1}\in
			Syl_{q}G_{1}$ and centralizes $x_{1}$ since $x\in C_G(Q).$ It follows that $[G_{1}:C_{G_{1}}(x_{1})]$
			is a power of $r$ which is not possible by \cite[Theorem 3.9]{isaacs} as $G_{1}$ is
			nonabelian simple.
			
			If furthermore $A$ is elementary abelian of order $p^{n}$ then by \cite{wa} we get
			that $h(G)\leq n.$

	\end{proof}
	
	\section{Proof of Theorem E}
	Let $A$ be a noncyclic abelian $p$-group of order $p^n$ acting on the solvable group $G$ by automorphisms so that the action is good. Suppose that there exists a natural number $m$ such that $[C_G(a),\, C_G(A)]_{m}=1$ for all $a\in A$. Then clearly $C_G(A)$ is nilpotent. Let $h=h(G)$. Due to good
	action there exists an irreducible $A$-tower $S_i,\, i=1,\ldots ,h$, that is, a
	sequence defined as in the proof of Theorem \ref{A}, of height $h$ in $G$. We may assume that $G=\prod_{i=1}^{h}S_i$.
	
	By the irreducibility of $P_1$ as an $A$-module, we have either $[P_1,A]=1$ or $[P_1,A]=P$. Suppose that $[P_1,A]=1$. Since  $P_2=\langle C_{P_2}(a)\,\,:\,\,1\ne a\in A\,\,\rangle$ by Lemma \ref{lem1}, we get $[P_2,S_1]=1$, which is impossible. Thus we may assume that there exists $1\ne a\in A$ such that $[P_1,a]\ne 1$. If $G$ is a $p'$-group, by Theorem 3.1 in \cite{tur} we see that $C_{P_h}(a),\ldots ,C_{P_2}(a)$ forms an $A$-tower. Set $A_1=\langle a\rangle$. By induction applied to the action of $A/A_1$ on the group $\prod_{i=2}^{h}C_{P_i}(a)$ we get $h-1\leq n-1$ and hence the theorem follows. On the other hand $p_i\ne p$ for each $i>1$ because otherwise we get $[P_2P_1,A]=1
	$ and so $[P_2,P_1]=1$. This forces that $p_1=p$ and so $[P_1,A]=1$, which is not possible. This completes the proof of Theorem \ref{E}.
\end{proof}

	Finally we state two more results that can easily be obtained by applying the same argument
	as in \cite{ja} by the use of Lemma \ref{lem1} and Theorem 4.5 in \cite{egj}.
	
	\begin{theoremn}
		Let $A$ be a noncyclic group of square free exponent $n$ acting on the group 
		$G$. Suppose that this action is good and that one of the following holds.
		
		$(1)$ There exists a natural number $m$ such that $[C_G(a), C_G(b)]_{m}=1$
		for all nonidentity elements $a, b\in A$.
		
		$(2)$ $Z(A) =1$ and $A$ has exponent $n$.
		
		\noindent Then $G$ is nilpotent of class bounded by a  function depending only on $m$ and $n$.
	\end{theoremn}

	\begin{theoremn}
		Let $A$ act on $G$ by automorphisms. Suppose that this action is good and
		that there exists a natural number $m$ such that $[C_G(a), C_G(b)]_{m}=1$
		for all nonidentity elements $a, b\in A$. If $G$ is not nilpotent, then $A$
		has the structure of the complement of some finite Frobenius group.
	\end{theoremn}

\end{document}